\documentclass[10pt]{amsart}
\usepackage{egastyle}

\usepackage{amssymb,amsmath,amscd,amsthm,amsfonts,wasysym,mathrsfs,amsxtra,stmaryrd,array}

\usepackage{graphicx,array,url}
\usepackage[vcentermath]{genyoungtabtikz}
\usepackage{tikz}
\usepackage[section]{placeins}
\usepackage{afterpage}
\usepackage{verbatim}
\usepackage{tikz-cd}
\input xy
\xyoption{all}

\newcommand{\PP}{\mathbb{P}}

\newcommand{\kernel}{\textnormal{ker}\,}

\newcommand{\Hom}{\textnormal{Hom}}

\newcommand{\Ext}{\textnormal{Ext}}

\newcommand{\arcot}{\textnormal{arcot}}

\newcommand{\Bc}{\mathcal{B}}

\newcommand{\Fc}{\mathcal{F}}
\newcommand{\Tc}{\mathcal{T}}

\newcommand{\Coh}{\mathrm{Coh}}

\newcommand{\arinj}{\ar@{^{(}->}}
\newcommand{\arsurj}{\ar@{->>}}
\newcommand{\areq}{\ar@{=}}

\newcommand{\ch}{\mathrm{ch}}

\newcommand{\Aut}{\mathrm{Aut}}
\newcommand{\oo}{{\overline{\omega}}}

\newcommand{\olB}{{\overline{B}}}
\newcommand{\bsm}{\begin{smallmatrix}}
\newcommand{\esm}{\end{smallmatrix}}

\newcommand{\mC}{\mathcal{C}}
\newcommand{\A}{\mathcal{A}}
\newcommand{\B}{\mathcal{B}}
\newcommand{\T}{\mathcal{T}}
\newcommand{\F}{\mathcal{F}}

\newcommand{\RHom}{\textnormal{RHom}}
\newcommand{\HN}{\text{HN}}

\newcommand{\cP}{\mathcal{P}}

\newcommand{\cO}{\mathcal{O}}

\newcommand{\oV}{\overline{V}}

\newcommand{\NS}{\text{NS}}

\newcommand{\im}{\text{im}}
\newcommand{\ST}{\text{ST}}
\newcommand{\STt}{\text{ST}_{\cO_C(t)}}
\newcommand{\STk}{\text{ST}_{\cO_C(-1)}}
\newcommand{\STCt}{\text{ST}_{\cO_{C_\nu}(t)}}
\newcommand{\Cone}{\text{Cone}}

\makeatletter
\newtheorem*{rep@theorem}{\rep@title}
\newcommand{\newreptheorem}[2]{%
\newenvironment{rep#1}[1]{%
 \def\rep@title{#2 \ref{##1}}%
 \begin{rep@theorem}}%
 {\end{rep@theorem}}}
\makeatother

  {\begin{genthm}{Question}{true}{#1}{paragraph}}%
  {\end{genthm}}

  {\begin{genthm}{Question}{true}{#1}{equation}}%
  {\end{genthm}}

\newreptheorem{theorem}{Theorem}

\usepackage{scalerel,stackengine}
\stackMath
\newcommand\reallywidehat[1]{%
\savestack{\tmpbox}{\stretchto{%
  \scaleto{%
    \scalerel*[\widthof{\ensuremath{#1}}]{\kern-.6pt\bigwedge\kern-.6pt}%
    {\rule[-\textheight/2]{1ex}{\textheight}}
  }{\textheight}%
}{0.5ex}}%
\stackon[1pt]{#1}{\tmpbox}%
}
\begin{document}

\title[]{Bridgeland/weak Stability Conditions under Spherical Twist Associated to a Torsion Sheaf}

\author[Tristan C. Collins]{Tristan C. Collins}
\address{Department of Mathematics \\
University of Toronto\\
40 St. George St.\\
 Toronto, ON \\
CA}
\email{tristanc@math.toronto.edu}

\author[Jason Lo]{Jason Lo}
\address{Department of Mathematics \\
California State University, Northridge\\
18111 Nordhoff Street\\
Northridge CA 91330 \\
USA}
\email{jason.lo@csun.edu}

\author[Yun Shi]{Yun Shi}
\address{Department of Mathematics \\
Trinity College\\
300 Summit Street\\
Hartford, CT, 06106\\
USA}
\email{yun.shi@trincoll.edu}

\author[Shing-Tung Yau]{Shing-Tung Yau}
\address{Yau Mathematical Sciences Center\\
Tsinghua University\\
Haidian District, Beijing\\
China}
\email{styau@tsinghua.edu.cn}


\begin{abstract}
In this paper, we study the action of an autoequivalence, the spherical twist associated to a torsion sheaf, on the standard Bridgeland stability conditions and a generalized weak stability condition on the derived category of a K3 surface. As a special case, we construct a Bridgeland stability condition associated to a non-nef divisor, which is outside  the geometric chamber, but conjecturally lies in the geometric component of the stability manifold by a conjecture of Bridgeland. We also briefly discuss the stability of certain line bundles at the weak stability condition associated to a nef divisor. 
\end{abstract}

\maketitle
\tableofcontents

\section{Introduction}
Bridgeland stability conditions, introduced by Bridgeland in \cite{StabTC}, provide a framework for defining stability conditions in triangulated categories. When the triangulated category is $D^b(X)$, the derived category of coherent sheaves on a smooth projective variety, Bridgeland stability conditions associated to an ample divisor have been constructed and extensively studied by many authors (e.g.\ \cite{SCK3, arcara2013minimal, BMT1} etc). A Bridgeland stability condition on $D^b(X)$ can be equivalently defined as a pair $(Z, \A)$, where  $\A$ is the heart of a bounded t-structure on $D^b(X)$, and $Z$ is a group homomorphism from the Grothendieck group of $\A$ to $\mathbb{H}$. It is shown in \cite{StabTC} that the set of stability conditions forms a complex manifold, on which the autoequivalence group $\Aut(D^b(X))$ acts on the left and the universal cover of $\text{GL}^+(2, \mathbb{R})$ acts on the right.
In the definition of Bridgeland stability condition, nonzero object in $\A$ can not be mapped to zero by the central charge $Z$. 
However, as one deforms the central charge towards certain limits, nonzero objects in the heart $\A$ may fall into the kernel of $Z$.
This situation gives rise to a notion called weak stability condition, originally defined by \cite{piyaratne2017stability}, and later on used and studied by \cite{BMT1, broomhead2022partial}. 
In \cite{CLSY1}, the authors further generalized the definition in \cite{piyaratne2017stability} by allowing more flexibility of the phases of the object in $\ker(Z)$, defining 
 their phases as limits of phases of Bridgeland stability conditions. In this way, the generalized weak stability conditions can be seen as degenerations or “limits” of Bridgeland stability conditions. 
 One can similarly study the group actions on the set of weak stability conditions defined in \cite{CLSY1}. 
 
 The study of group actions on Bridgeland/weak stability conditions has many important applications. For example, in \cite{CLSY2}, the authors study the action of a specific autoequivalence--the relative Fourier-Mukai transform--on certain Bridgeland/weak stability conditions, and use it to show the stability of line bundles in certain region of the stability manifold. In this application, understanding the image of Bridgeland/weak stability conditions under the action of the relative Fourier-Mukai transform has been demonstrated to be very useful. From the perspective of applications along this line, it is also evident that the negative curves on a surface play a significant role in studying the stability of line bundles, which motivates the exploration of the weak stability condition at the boundary corresponding to negative curves.  
Motivated by this application, in this paper, we focus on the action of a different type of autoequivalence, the spherical twist associated with a torsion sheaf supported on a negative curve. 
Let $X$ be a K3 surface, we study the action of this spherical twist functor on the standard Bridgeland stability conditions associated to a ample divisor and certain weak stability conditions associated to a nef divisor, and give explicit construction of the images. As a special case of the general construction, we give the first construction of a Bridgeland stability condition associated to a divisor outside the closure of the ample cone, so that the stability condition is outside  the geometric chamber, but  conjecturally lies in the geometric component of the stability manifold by a conjecture of Bridgeland \cite[Conjecture 1.2]{SCK3}. Finally, we discuss the stability of line bundles with respect to certain weak stability conditions on specific K3 surfaces.  Although this is not a direct application of the behavior of stability conditions under spherical twists, understanding the stability of line bundles in weak stability conditions has significant applications to the study of the stability of the line bundles at the standard Bridgeland stability conditions, so we include a brief discussion here.

\subsection{Outline}
Since the image of the weak stability condition associated to a nef divisor under spherical twist is easier to describe, we start with the study of weak stability conditions.  
In Section 2, we review the basic definitions and properties of generalized weak stability conditions, including an example associated with a nef divisor, as studied in \cite{CLSY1}.
We also review the definition and some basics of spherical twist functor on $D^b(X)$ for $X$ a K3 surface. In Section 3, we define another weak stability condition associated to a nef divisor. In Section 4, we compute the image of a central charge under the spherical twist functor, and show that the weak stability conditions constructed in Section 3 is the image of the weak stability condition presented in Section 2 under the spherical twist. In Section 5, we give a construction of the image of a standard form Bridgeland stability condition associated to an ample divisor under the spherical twist. Finally, Sections 6 briefly discuss some results on the stability of line bundles with respect to the weak stability condition in Section 2.
\subsection{Acknowledgment}
TCC was partially supported by NSF CAREER grant DMS-1944952. JL was partially supported by NSF grant DMS-2100906.   YS was supported by the AMS-Simons travel grant. 

\section{Preliminaries}
In this Section, we review some basic Definitions and notions of weak stability condition developed in \cite{CLSY1}.
\subsection{Weak stability condition}
We begin with some basic Definitions. 

	\begin{defn}
		\label{Def:Weak}
		A weak stability condition on $\mathcal{D}$ is a triple $$\sigma=(Z, \A, \{\phi(K)\}_{K\in \ker(Z)\cap \A}),$$ where $\mathcal{A}$ is the heart of a bounded t-structure on $\mathcal{D}$, and $Z: K(\mathcal{D})\rightarrow \mathbb{C}$ a group homomorphism satisfying:
		
		(i) For any $E\in \mathcal{A}$, we have $Z(E)\in\mathbb{H}\cup\mathbb{R}_{\leq 0}$. For any $K\in \ker(Z)\cap\A$, we have $0<\phi(K)\leq 1$. 
  
		(ii) (Weak see-saw property) For any short exact sequence 
		\begin{equation*}
			0\rightarrow K_1\rightarrow K\rightarrow K_2\rightarrow 0
		\end{equation*}
		in $\ker(Z)\cap\A$, we have $\phi(K_1)\geq \phi(K)\geq \phi(K_2)$ or $\phi(K_1)\leq \phi(K)\leq \phi(K_2)$.

		 For any object $E\notin \ker(Z)$, define the phase of an object $E\in\A$ by $$\phi(E)=(1/\pi) \mathrm{arg}\, Z(E)\in (0, 1].$$ We further require 
   
   (iii) The phase function satisfies the Harder-Narasimhan (HN) property. 
	\end{defn}
 
Define the slope of an object with respect to a weak stability condition by 
$\rho(E)=-\text{cot}(\pi\phi(E))$
Following \cite{piyaratne2015moduli}, an object $E\in \A$ is (semi)-stable if for any nonzero subobject $F\subset E$ in $\A$, we have 
\[
\rho(F)<(\leq)\rho(E/F),
\]
or equivalently,
\[
\phi(F)<(\leq)\phi(E/F).
\]
 Similar to Bridgeland stability conditions, weak stability can be characterized by the notation of slicing as in Proposition 3.5 in \cite{CLSY1}.

\subsection{Weak stability condition associated to a nef divisor I}

In 5.13 of \cite{CLSY1}, we constructed a weak stability condition associated to a particular nef divisor on a Weierstra{\ss}  elliptic K3 surface, following the construction of Bridgeland stability condition associated to a nef divisor in \cite{tramel2017bridgeland}. The construction in 5.13 of \cite{CLSY1} actually works for a more general class of nef divisor on a general K3 surface, so we write it in more generality. 

Let $X$ be a K3 surface with a smooth rational curve $C$ such that $C^2=-2$. 
Assume that there exist a nef divisor $\nu$ such that $\nu\cdot C=0$ and $\nu\cdot C'>0$ for any other irreducible curve $C'$ which is not contained in $C$.  
We consider the central charge $Z_{V, \nu}$ which is defined by 
\[
Z_{V, \nu, B}(E)=-\ch^B_2(E)+V\ch^B_0(E)+i\nu\cdot \ch^B_1(E).
\]
If $B=0$, we omit $B$ in the notation. Recall the construction of heart in Tramel-Xia \cite{tramel2017bridgeland}. 
Set
\begin{align*}
  \Tc^s_\nu &= \langle E \in \Coh (X) : E \text{ is $\mu_\nu$-semistable}, \mu_{\nu, \mathrm{min}}(E)>s \rangle \\
  \Fc^s_\nu &= \langle E \in \Coh (X) : E \text{ is $\mu_\nu$-semistable}, \mu_{\nu, \mathrm{max}}(E) \leq s \rangle.
\end{align*} 
Let 
\[
\Coh(X)^{\nu, s}=\langle \F^s_\nu[1], \T^s_\nu\rangle,
\]
when $s=0$ we omit $s$ in the notation.
Denoting the category  $\langle \cO_{C}(j)|j\leq k\rangle$ by $\mC_{k}$, and set
\[
\begin{split}
  \Fc^1_{\nu,k} &= \mC_{k}. \\
    \Tc^1_{\nu,k} &= \{ E \in \Coh(X)^{\nu, s} : \Hom (E, \mC_{k}) = 0\}.
    \end{split}
\]

Then \cite[Lemma 3.2]{tramel2017bridgeland} implies that $(\Tc^1_{\nu,k}, \Fc^1_{\nu,k})$ is a torsion pair in $\Coh^{\nu, s}$, tilting at this torsion pair we obtain the heart
\[
  \Bc^s_{\nu,k} = \langle \Fc^1_{\nu,k}[1], \Tc^1_{\nu,k} \rangle,
\]
when $s=0$, we omit $s$ in the notation. 
 
The category $\B_{\nu, k}$ can be equivalently defined by
\begin{prop}[\cite{CLSY2} Proposition 10.6]
\label{prop:defB0k}
	$\B_{\nu, k}$ is obtained by tilting $\Coh(X)$ at the torsion pair
	\begin{equation*}
		\begin{split}
			\T_{\nu, k}&:=\{E\in \Coh(X)|E\in \T^0_\nu \text{ and } \Hom_{\Coh^\nu}(E, \mC_{ k})=0\}\\
			\F_{\nu, k}&:=\{E\in \Coh(X)|E\in \langle\F^0_\nu, \mC_k\rangle\}.
		\end{split}
	\end{equation*}
	\end{prop}
 
Since $X$ is a K3 surface, we still have the following Proposition in \cite{CLSY1}.
\begin{prop}[\cite{CLSY1} Proposition 5.14]
	We have $Z_{V, \nu}(\B_{\nu, k})\in \mathbb{H}_0$ for $k=-1, -2$. 
Furthermore, we have	
	when $k=-1$,
	$$\kernel (Z_{V, \nu})\cap \B_{\nu, -1}=\langle\cO_C(-1)[1]\rangle,$$
	when $k=-2$,
	$$\kernel (Z_{V, \nu})\cap \B_{\nu, -2}= \langle \cO_C(-1)\rangle.$$
\end{prop}

To define a weak stability condition, we still need to define the phase of the objects in the kernel of the central charge. We only consider the case where $k=-2$.
Recall that if $\omega$ is an ample class, then $\nu_\epsilon:=\nu+\epsilon\omega$ is ample for $\epsilon>0$. Picking an arbitrary ample class $\omega$, we would like to define weak stability conditions $\sigma^b_{V, \nu}=(Z_{V, \nu}, \B_{\nu, -2}, \{\phi^b_{V, \nu}(K)\}_{K\in\ker(Z_{V, \nu})\cap\B_{\nu, -2}})$ as the limit of the standard Bridgeland stability conditions $\sigma_{V, \nu_\epsilon}=(Z_{V, \nu_\epsilon}, \Coh^{\nu_\epsilon})$.
In particular, we define the phases $\{\phi^b_{V, \nu}(K)\}_{K\in \ker(Z_{V, \nu})\cap \B_{\nu, -2}}$ by taking the limit of the phases of $\sigma_{V, \nu_\epsilon}(K)$ as $\epsilon\to 0$.
Note that
$$Z_{V, \nu_\epsilon}(\cO_C(-1))=i\epsilon\omega\cdot C.$$
Since 
$\ker(Z_{V, \nu})\cap \B_{\nu, -2}=\langle \cO_{C}(-1)\rangle$. Then for any $K\in \ker(Z_{V, \nu})\cap \B_{\nu, -2}$, we have
\[
\lim\limits_{\epsilon\rightarrow 0}\phi_{V, \nu_\epsilon}(K)=\frac{1}{2}.
\]
From the above discussion, for any $K\in\ker(Z_{V, \nu})\cap \B_{\nu, -2}$ we define $\phi^b_{V, \nu}(K)=\frac{1}{2}$. 
 Then the weak see-saw property follows automatically.
 Since $X$ is a K3 surface, the torsion sheaf $\cO_{C}(t)$ is a spherical object for any $t\in \mathbb{Z}$, in particular $\Ext^1(\cO_C(t), \cO_C(t))=0$. Then the argument in Proposition 5.8 in \cite{CLSY1} implies that the triples $\sigma^b_{V, \nu}$   
 satisfies the HN property, hence define a weak stability condition. 

\begin{rem}
\label{rk:Brinub}
Note that the above definition of the phases of objects in $\ker(Z_{V, \nu})$ as the limits of Bridgeland stability conditions is not unique. E.g. let $B\in \NS_\mathbb{Q}(X)$ such that $B\cdot\nu=0$, (e.g. taking $B=uC$). Then 
\cite{tramel2017bridgeland} shows that the pair $(Z_{V, \nu, B}, \B_{\nu, -2})$ defines a Bridgeland stability condition for $-1<C\cdot B<0$ ($0<u<\frac{1}{2}$). We denote this set of Bridgeland stability conditions by $\sigma_{V, \nu, B}$. 

Let $p=[\epsilon, u]\in \mathbb{P}^1$. For $K\in \ker(Z_{V, \nu})\cap \B_{\nu, -2}$, we can define 
\[
\phi^{b, p}_{V, \nu}(K)=\lim\limits_{p=[\epsilon, u],\epsilon\rightarrow 0}\phi_{V, \nu_\epsilon, uC}(K)=\frac{1}{\pi}\arcot(kp),
\]
where 
$k=\frac{2}{\omega\cdot C}$.
Then we obtain a family of weak stability conditions $\sigma^{b,p}_{V, \nu}$. Let $\omega$ be an arbitrary ample divisor on $X$. The weak stability condition corresponds to $p=[1: 0]$ is the limit of $\sigma_{\nu+\epsilon\omega, 0}$ as $\epsilon\rightarrow 0$, i.e. $\sigma^b_{V, \nu}$. 
While the weak stability function corresponds to $p=[0:1]$ is the limit of $\sigma_{V, \nu, uC}$ as $u\rightarrow 0$. 
\end{rem}

\subsection{spherical twist functor}
Spherical twist functor is defined in \cite{ST01}. Let $X$ be a projective variety of dimension $n$. Recall that an object $E$ in $D^b(X)$ is called spherical if 
$\Ext^i(E, E)\simeq \mathbb{C}$ when $i=0, n$, and equals to $0$ otherwise, and $E\otimes \omega_X\simeq E$. 
\begin{defn}[\cite{ST01}]
\label{def:ST}
Let $E$ be a spherical object in $D^b(X)$. The spherical twist functor $\ST_E:D^b(X)\to D^b(X)$ is defined to be the Fourier-Mukai transform with kernel equals to 
\[\text{Cone}(E^\vee\boxtimes E\to \cO_\Delta).
\]
\end{defn}
By Exercise 8.5 in \cite{huybrechts2006fourier}, the image of an object $F$ under the spherical twist functor is given by 
\[
\ST_E(F)\simeq \text{Cone}(\RHom(E, F)\otimes E\to F).
\]
Theorem 1.2 in \cite{ST01} implies that the functor $\ST_E$ associated to a spherical object $E$ is an exact self-equivalence of $D^b(X)$.

\section{Weak stability condition associated to a nef divisor II}
Let $X$, $C$ and $\nu$ be as in the previous Section. In this Section, we construction a different weak stability condition associated to the nef divisor $\nu$.
Recall the heart $\B_{\nu, -2}$ from the previous Section. Let $\T_C$ be the subcategory $\langle\cO_C(-1)\rangle$ of $\B_{\nu, -2}$, and let 
\[
\F_C:=\{F\in \B_{\nu, -2}|\Hom(\T_C, F)=0\}.
\]
\begin{prop}
\label{prop:torsionpair2nef}
The pair $(\T_C, \F_C)$ defines a torsion pair on $\B_{\nu, -2}$.
\end{prop}
\begin{proof}
By definition, we have $\Hom(\T_C, \F_C)=0$. We only need to show for any object $E\in \B_{\nu, -2}$, $E$ fits in a s.e.s. 
\[
0\to T\to E\to F\to 0,
\]
where $T\in \T_C$ and $F\in \F_C$. If $\Hom(\cO_C(-1), E)=0$, we have $E\in \F_C$. From now on we assume $\Hom(\cO_C(-1), E)\neq 0$. Then there exists a nonzero map $\theta_1: \cO_C(-1)\to E$, and we have the following s.e.s's in $\B_{\nu, -2}$:
\begin{equation}
\label{eq:3.1ses}
\begin{split}
0&\to K_1\to\cO_C(-1)\to I_1\to 0\\
0&\to I_1\to E\to N_1\to 0
\end{split}
\end{equation}
where $K_1$, $I_1$ and $N_1$ are the kernel, image and cokernel of $\theta_1$ in $\B_{\nu, -2}$ respectively. 
We show that $K_1= 0$. 
Assume not, we have s.e.s in $\Coh(X)$
\[
0\to H^{-1}(I_1)\to K_1\to \cO_C(j)\to 0
\]
for some $j\leq -1$. If $H^{-1}(I_1)=0$, then $H^0(I_1)\neq 0$ and the first s.e.s. in equation \ref{eq:3.1ses} is a s.e.s. in $\Coh(X)$, this implies that $K_1\in \mC_{-2}$ which is a contradiction. Then $H^{-1}(I_1)\neq 0$. Then $\mu_\nu(K_1)\geq 0$ and $\mu_\nu(H^{-1}(I_1))\leq 0$ forces
$\mu_\nu(K_1)= 0$, hence $K_1$ is a torsion sheaf supported on $C$. This implies that $H^{-1}(I_1)$ is a torsion sheaf supported on $C$ and hence $H^{-1}(I_1)\in \mC_{-2}$. Since $\Hom(K_1, 
\mC_{-2})=0$ we must have $j=-1$.
But GRR implies that $\Ext^1(\cO_C(-1), \cO_C(-2))=0$, also $\Hom(K_1, \mC_{-2})=0$, hence $\Hom(H^{-1}(I_1), \cO_C(-2))=0$. This contradicts $H^{-1}(I_1)\in \mC_{-2}$.

Hence $K_1=0$, and we have 
\[
0\to \cO_C(-1)\to E\to N_1\to 0.
\]
If $\Hom(\cO_C(-1), N_1)=0$, this is the s.e.s we need. If not, we repeat the process for $N_1$. Since $\Ext^1(\cO_C(-1), \cO_C(-1))=0$, the dimension of $\Hom(\cO_C(-1), N_i)$ decreases and eventually we obtain the desired s.e.s. 
\end{proof}
Let $\A_{\nu, -2}:=\langle\F_C, \T_C[-1]\rangle$. Since $\ker Z_{V, \nu}\cap \A_{\nu, -2}=\langle\cO_C(-1)[-1]\rangle$, the central charge $Z_{V, \nu}$ also defines a weak stability function on $\A_{\nu, -2}$. 
We define the phases of the object $\cO_C(-1)[-1]$ to be $\phi^a_{V, \nu}(\cO_C(-1)[-1])=\frac{1}{2}$. Since $\ker(Z_{V, \nu})\cap\A_{\nu, -2}=\langle\cO_C(-1)[-1]\rangle$, then any object $K\in\ker(Z_{V, \nu})\cap\A_{\nu, -2}$ has phase equal to $\frac{1}{2}$, and the weak see-saw property is automatically satisfied. 
\begin{prop}
The triple $\sigma^a_{V, \nu}:=(Z_{V, \nu}, \A_{\nu, -2}, \{\phi^a_{V, \nu}(K)\})$ satisfies the HN property, hence defines a weak stability condition. 
\end{prop}
\begin{proof}
The proof is almost the same as Proposition 5.8 in \cite{CLSY1}. The only difference is in verifying (ii) in Proposition 3.6 in \cite{CLSY1}. Assuming we have a sequence 
\[
E=E_1\twoheadrightarrow E_2\twoheadrightarrow E_3\twoheadrightarrow...
\]

Consider the short exact sequences 
\begin{equation}\label{Equ:HNKEE}
0\to K_i\to E_i\to E_{i+1}\to 0.
\end{equation}
Then $E_i$ and $E_{i+1}$ fit into a commutative diagram
\[
\begin{tikzcd}
    0\ar{r} &F_i\ar{r}\ar{d}& E_i\ar{r}\ar{d}& T_{i}\ar{r}\ar{d}& 0\\
    0\ar{r} &F_{i+1}\ar{r}& E_{i+1}\ar{r}\ar{d}& T_{i+1}\ar{r}\ar{d}& 0\\
   & &0 & 0&
\end{tikzcd}
\]
where $F_i$ and $F_{i+1}$ are in $\F_C$, $T_i$ and $T_{i+1}$ are in $\T_C[-1]$. 
Hence $T_i\simeq \cO_C(-1)[-1]^{n_i}$. Since $T_i\twoheadrightarrow T_{i+1}$, we have $n_i$ is not increasing. As a result, skipping finite number of terms, we may assume $T_i\simeq T_{i+1}$ for all $i> 0$. 
Then it is enough to show a sequence $F_i\twoheadrightarrow F_{i+1}\twoheadrightarrow ... $ in $\F_C$ with $\phi^a_{V, \nu}(K_i)>\phi^a_{V, \nu}(F_{i+1})$ stabilizes.
The rest of the argument is the same as Proposition 5.8 in \cite{CLSY1}. 
\end{proof}

\section{Weak stability condition under $\ST_{\cO_C(-1)}$}
\subsection{Central charge formula}
Let $X$, $C$ be as in the previous Section. Recall that we denote $e:=-C^2=2$.
It is well known that $\cO_C(t)$ is a spherical object for any $t\in \mathbb{Z}$. We denote the spherical twist functor with respect to $\cO_C(t)$ by $\STt$. In this Section, we solve a central charge formula under the action of $\STt$. The coordinates we are using is not in the most general form, nonetheless, this coordinate will make certain discussion in Section 5, e.g. Remark \ref{rk:5.9nonnef} more natural. 

Let $D$ be a 
divisor on $X$ such that $D\cdot C=1$ and $D^2=0$. E.g. let $\nu$ be as in the previous Section, such that $\nu^2=2$, then $D:=\frac{1}{e}(\nu-C)$ satisfies the above conditions. Let $\psi, \chi\in \langle C, D\rangle^\perp$. Let $\omega$, $B$, $\oo$, $\olB$ be classes in $\NS_\mathbb{R}(X)$ which are of the form
\[
\begin{split}
    \omega&=C+(e+D_\omega)D+G_\omega\psi,\\
    B&=R_B(C+(D_B+e)D+G_B\chi),
\end{split}
\]
where $D_\omega>-1$, $R_B$, $D_B, G_B\in \mathbb{R}$. We use similar notation for the expression of $\oo$ and $\olB$ in this coordinate and solve the central charge formula
\begin{equation}
\label{eq:centralchargefor}
Z_{V, \omega, B}(\STt(E))=g\cdot Z_{\overline{V}, \overline{\omega}, \overline{B}}(E)
\end{equation}
for some element $g\in \text{GL}_n$.

Denote the Mukai vector of $E$ by $\nu(E)=(n, c_1(E), s)$. Then we have
\[
\nu(\STt(E))=(n, c_1(E), s)+(c-n(t+1))(0, C, t+1).
\]
Denoting $c:=\ch_1(E)\cdot C$, $d=\ch_1(E)\cdot D$, then we have the following numerical invariant for $\STt(E)$. 
\[
\begin{split}
n(\STt(E))&=n, \\
c(\STt(E))&=-c+2n(t+1), \\
d(\STt(E))&=d+c-n(t+1), \\
s(\STt(E))&=s+(c-n(t+1))(t+1).
\end{split}
\]
Then \[
\begin{split}
\Re Z_{V, \omega, B}(\STt(E))&=
-s+(R_B(D_B+1)-(t+1))c+R_B(D_B+e)d\\
&+((t+1)^2-R_BD_B(t+1)+V+1-\frac{1}{2}B^2)n
+R_BG_B\chi c_1(E)\\
\Im Z_{V, \omega, B}(\STt(E))&=(D_\omega+1)c+(D_\omega+e)d-(D_\omega(t+1)+\omega\cdot B)n+G_\omega\psi c_1(E).
\end{split}
\]

Before the spherical twist, we have  
\[
\begin{split}
\Re Z_{\overline{V}, \oo}(E)&=-s+R_{\olB}c+R_{\olB}(D_{\olB}+e)d+(\oV+1-\frac{1}{2}\olB^2)n+G_{\olB}c_1(E)\overline{\chi},\\
\Im Z_{\overline{V}, \oo}(E)&=c+(D_\oo+e)d-n\oo\cdot\olB+G_\oo\overline{\psi}c_1(E).
\end{split}
\]
Using $\begin{pmatrix}
		1 & -\frac{R_B(D_B+e)}{D_\omega+e}\\ 
		0 & 1
	\end{pmatrix}$ we can cancel the term involving $d$ in $\Re Z_{V, \omega, B}(\STt(E))$, 
then the coefficient of $s$ and $d$ implies that 
 \[
 g=\begin{pmatrix}
		1 & 0\\ 
		0 & \frac{D_\omega+e}{D_\oo+e}
	\end{pmatrix}.
 \]
 Then coefficient of $c$ implies that 
 \[
 D_\omega=-\frac{D_\oo}{D_\oo+1}.
 \]
As a special case of solving $\omega$ and $B$ from $\oo$, $\olB$, we have the following Proposition. 
\begin{prop}
\label{prop:cencharfor}
Assume that $\overline{\psi}=0$, then $\psi=0$. Let $g=\begin{pmatrix}
		1 & 0\\ 
		0 & \frac{1}{D_\oo+1}
	\end{pmatrix}.$ 

When $t=-1$, we have 
\[
Z_{V, \omega, 0}(\STt(E))=g\cdot Z_{V, \oo, 0}(E)
\]

When 
$D_\oo=0$, we have
\[
Z_{V, \omega, -(t+1)C}(\STt(E))=g\cdot Z_{V, \omega, 0}(E)
\]

When $t= -1$ 
$D_\oo=0$, and $\overline{B}=\overline{u}C$.
\[
Z_{V, \omega, uC}(\STt(E))=g\cdot Z_{V, \omega, \overline{u}C}(E)
\]
where $u=-\overline{u}$.
\end{prop}
\begin{proof}
This follows from direct computation by comparing the coefficients of $c$ and $n$. 
\end{proof}

\subsection{Correspondence of hearts}
In this subsection, we check that the correspondence of hearts under $\STk$. 
The following result is stated and used in \cite{tramel2017bridgeland}, we give the proof here just for completion. 
\begin{lem}
\label{lem:Ck}
Let $\mC:=\langle\cO_C(j)\vert j\in \mathbb{Z}\rangle$. Assume that $E\in \mC$, and $\Hom(\cO_C(k), E)=0$, then 
$E\in \mC_{k-1}$.
\end{lem}
\begin{proof}
We proceed by induction. Let $E\in \mC$, then $\ch_1(E)=nC$ for some $n\in \mathbb{Z}$. 

The statement is obviously true when $n=1$. Now we assume that the statement is true for $n\leq l$. 
Assume that $\ch_1(E)=(l+1)C$. 
Let $i:C\hookrightarrow X$. Note that $i_*i^*E\simeq F\oplus Q$ where  $0\neq F\simeq \oplus_p \cO_C(n_p)$ and $Q$ is supported in dimension $0$. We take $E\to F$ and complete it into a s.e.s. in $\Coh(X)$:
\[
0\to G\to E\to F\to 0.
\]
If $G=0$, then $E\simeq F$, then $\Hom(\cO_C(k), E)=0$ implies $n_p\leq k-1$ for all $p$. Hence $E\in \mC_{k-1}$. From now on we assume $G\neq 0$. Since $i^*E\simeq i^*i_*i^*E$, we have 
\begin{equation}
\label{eq:5.1surj}
    L^{-1}i^*(F)\rightarrow i^*G\to i^*Q\to 0.
\end{equation}
 Since $\Hom(\cO_{C}(k), G)=0$, and $\ch_1(G)= rC$ for some $r\leq l$, the induction hypothesis implies that $G\in \mC_{k-1}$. 
Then there exists a s.e.s. 
\[
0\to G_1\to G\to \cO_C(q)\to 0
\]
for some $q\leq k-1$. 
Hence if we write $i^*G\simeq \oplus_i\cO_C(m_i)\oplus R$ where $R$ supported in dimension $0$, then there exists $i_0$ such that $m_{i_0}\leq k-1$.
Direct computation shows that $L^{-1}i^*(F)\simeq \oplus\cO_C(n_p+2)$.
Then equation \ref{eq:5.1surj} implies that there exists $p_0$ such that $n_{p_0}+2\leq k-1$. 
Take the s.e.s.
\[
0\to K\to E\to \cO_C(n_{p_0})\to 0,
\]
the induction hypothesis implies that $K\in \mC_{k-1}$, and hence $E\in \mC_{k-1}$.

\end{proof}
We abbreviate notation and denote $\STk$ by $\ST$.
 \begin{thm} 
\label{thm:corrheart}
Using the above notation, we have $\ST(\B_{\nu, -2})=\A_{\nu, -2}$.
\end{thm}
\begin{proof}
Since $\B_{\nu, -2}$ and $\A_{\nu, -2}$ are both hearts of t-structures in $D^b(X)$, it is enough to show that $\ST(\B_{\nu, -2})\subseteq \A_{\nu, -2}$.

 We first consider the case when $G\in \T_{\nu, -2}$. Denote dim $\Ext^i(\cO_C(-1), G)$ by $n_i$. Then by the definition of spherical twist and taking l.e.s. of cohomology, we have 
\[
\begin{split}
0\to H^{-1}(\ST(G))\to \cO_C(-1)^{n_0}&\to G\to H^0(\ST(G))\to \cO_C(-1)^{n_1}\to 0,\\
H^1(\ST(G))&\simeq \cO_C(-1)^{n_2}.
\end{split}
\]
Since $G\in \T_{\nu, -2}$, any quotient of $G$ in $\Coh(X)$ is also in $\T_{\nu, -2}$, hence $H^0(\ST(G))\in \T_{\nu, -2}\subseteq \B_{\nu, -2}$. By definition of $n_0$, we have $\Hom(\cO_{C}(-1), H^{-1}(\ST(G)))=0$. Since $H^{-1}(\ST(G))$ is a subsheaf of $\cO_{C}(-1)^{n_0}$, Lemma \ref{lem:Ck} implies that $H^{-1}(\ST(G))\in \mC_{-2}$. Then $\ST(G)^{\leq 0}\in \B_{\nu, -2}$.
Note that $\ST(G)$ fits into an exact triangle
\[
\ST(G)^{\leq 0}\to \ST(G)\to H^1(\ST(G))[-1]\xrightarrow{+1}.
\]

Applying $\Hom(\cO_C(-1), \_)$, we have $\Hom^{-1}(\cO_C(-1), H^1(\ST(G))[-1])=0$, and also
\[
\Hom(\cO_C(-1), \ST(G))=\Hom(\ST^{-1}(\cO_C(-1)), G)=\Hom(\cO_C(-1)[1], G)=0.
\]
This implies that $\Hom(\cO_C(-1), \ST(G)^{\leq 0})=0$, and hence $\ST(G)^{\leq 0}\in \F_C$. Also, the l.e.s. of cohomology implies that $H^1(\ST(G))[-1]\in \T_C[-1]$, hence we have $\ST(G)\in \A_{\nu, -2}$.

Next we consider when $G=F[1]\in \F_{\nu, -2}[1]$. 
The l.e.s. of cohomology becomes
\[
\begin{split}
0\to F\to H^{-1}(\ST(G))&\to \cO_C(-1)^{n_0}\to 0\\
H^0(\ST(G))&\simeq \cO_C(-1)^{n_1}.
\end{split}
\]
Then $H^0(\ST(G))\in\B_{\nu, -2}$. 
Consider $H^{-1}(\ST(G))$. Note that the s.e.s in $\Coh(X)$ above gives an exact triangle 
\[
\cO_{C}(-1)^{n_0}\xrightarrow{\iota} F[1]\to H^{-1}(\ST(G))[1]\xrightarrow{+1},
\]
where $\cO_{C}(-1)^{n_0}$and $F[1]$ are objects in $\B_{\nu, -2}$. Then the same argument in Proposition \ref{prop:torsionpair2nef} implies that $\iota$ is injective in $\B_{\nu, -2}$. 
This implies that $H^{-1}(\ST(G))[1]\in \B_{\nu, -2}$.
Next we compute
\[
\Hom(\cO_C(-1), \ST(G))=\Hom(\ST^{-1}(\cO_C(-1)), G)=\Hom(\cO_C(-1)[1], G).
\]
Since $F=G[-1]\in \langle\mC_{-2}, \F_\nu^0\rangle$, we have $\Hom(\cO_C(-1)[1], F[1])=0$. 
Hence $\Hom(\cO_C(-1), \ST(G))=0$ and $\ST(G)\in \F_{C}\subseteq \A_{\nu, -2}$.

Since any object $E\in \B_{\nu, -2}$ fits into a s.e.s. 
\[
0\to H^{-1}(E)[1]\to E\to H^0(E)\to 0,
\]
the above discussion showed that $\ST(H^{-1}(E)[1])\in \A_{\nu, -2}$ and $\ST(H^0(E))\in\A_{\nu, -2}$, hence $\ST(E)\in \A_{\nu, -2}$.
\end{proof}

\begin{rem}
Note that the stability conditions $\sigma^b_{V, \nu}$ and $\sigma^a_{V, \nu}$ are indeed different weak stability conditions. By the equivalent definition of slicing, it is enough to show there exists object in $\cP_b(\phi)$ for some $\phi$, that is not semistable in $\sigma^a_{V, \nu}$. Consider $\cO_C(t)$ for $t\geq 0$. Since such objects have maximal phase in $\B_{\nu, -2}$, we have $\cO_C(t)\in \cP_b(1)$ for any $t\geq 0$. On the other hand, if $\cO_C(t)[i]\in \A_{\nu, -2}$ for some $i$. Then it fits into a s.e.s. 
\[
0\to F\to \cO_C(t)[i]\to T\to 0,
\]
where $F\in \F_C$ and $T\in \T_C[-1]$. Then the l.e.s. of cohomology implies that $H^{-1}(F)\simeq H^{-1}(\cO_C(t)[i])$, $H^{0}(F)\simeq H^{0}(\cO_C(t)[i])$ and $H^1(\cO_C(t)[i])\simeq H^1(T)$. This forces $H^1(\cO_C(t)[i])=0$ and hence $T=0$. But since $\cO_C(t)[i]\notin \F_C$ for any $i$, we have $\cO_C(t)[i]\notin \A_{\nu, -2}$ for any $i$. Hence the two weak stability conditions are not the same. 
\end{rem}

\section{Bridgeland stability condition under $\ST_{\cO_C(-1)}$}

Let $X$ be a K3 surface, $\nu_a\in \NS_\mathbb{Q}(X)$ be an arbitrary ample class.  Let $\sigma_{V, \nu_a}=(Z_{V, \nu_a}, \Coh^{\nu_a})$ be the standard Bridgeland stability condition associated to $\nu_a$. In this Section, we study the Bridgeland stability condition $\ST_{\cO_C(-1)}(\sigma_{V, \nu_a})$. 
A particularly interesting example is as follows. Let $C$ and $\nu$ be as in the previous Section, taking $D=\frac{1}{e}(\nu-C)$, then $\nu$ can be written as $\nu=C+eD$ with $D\cdot C=1$. Assume $D$ is a nef divisor. Let $\nu_a=\nu+aD$. If $a>0$, $\nu_a$ is ample. If $a<0$, then $\nu_a\cdot C<0$, hence $\nu_a$ is not nef. Then the construction in this Section in particular constructs a Bridgeland stability condition associated to a non-nef divisor. This is our original motivation, so we keep the notation of the ample class to be $\nu_a$, but the construction in this Section works for an arbitrary ample class in $\NS_{\mathbb{Q}}(X)$ such that there exists a solution to the central charge formula \ref{eq:centralchargefor} in Section 4.

Denoting the standard torsion pair in $\Coh(X)$ w.r.t. the ample class $\nu_a$ by $(\T^0_{\nu_a}, \F^0_{\nu_a})$.  
We denote the standard tilted heart by $\Coh^{\nu_a}:=\langle\F^0_{\nu_a}[1], \T^0_{\nu_a}\rangle$. Define the set 
\[
S^a_{k}:\{E\in \Coh^{\nu_a}\vert E\hookrightarrow \cO_C(k)\}.
\]
Let $\mC^a_{k}$ be $\langle E\in S^a_{k}\rangle$.
Let
 \[
\begin{split}
\T^1_{\nu_a, k}&=\{T\in \Coh^{\nu_a}\vert \Hom(T, \mC_k^a)=0\}\\
\F^1_{\nu_a, k}&=\mC^a_{k}.
\end{split}
\]
Using standard argument, we have the following simple Lemma. 
\begin{lem}
\label{lem:subclose}
Let $\A$ be an abelian category, and $\A'$ a subcategory of $\A$ closed under extension. Let 
\[
0\to E_1\to E\to E_2\to 0
\]
be a s.e.s. in $\A$. 

(i) Assume any subobject of $E_i$ is an objects in $\A'$, then any subobject of $E$ is also an object in $\A'$.

(ii) Assume any quotient of $E_i$ is an objects in $\A'$, then any quotient object of $E$ is also an object in $\A'$.
\end{lem}

\begin{prop}
Let $\nu_a$ be an ample class in $\NS_\mathbb{Q}(X)$. The pair $(\T^1_{\nu_a, k}, \F^1_{\nu_a, k})$ defines a torsion pair on $\Coh^{\nu_a}$ for all $k\in \mathbb{Z}$.  
\end{prop}
\begin{proof}
The argument is standard. It is enough the show the following two claims. 

(i) $\F^1_{\nu_a, k}$ is closed under taking subobject in $\Coh^{\nu_a}$. 

(ii) If there is a sequence of inclusions $...\hookrightarrow E_i\hookrightarrow E_{i-1}\hookrightarrow...\hookrightarrow E_1\hookrightarrow E$ in $\Coh^{\nu_a}$ with $E_i/E_{i+1}\in \F^1_{\nu_a, k}$ for all $i$, then for $i\gg 0$ we have $E_i\simeq E_{i+1}$.

Claim (i) follows directly from induction on Lemma \ref{lem:subclose}. 

For (ii), assume there is such a sequence of inclusions. Taking long exact sequence of cohomology, we have 
$H^{-1}(E_i)\simeq H^{-1}(E_{i+1})$ for all $i$. So we may assume the $E_i$'s in the sequence are all in $\T^0_{\nu_a}$. Since $\F^1_{\nu_a, k}\subseteq \T^0_{\nu_a}$, we have $E_i/E_{i+1}\in \T^0_{\nu_a}$. Hence $\ch_1(E_i/E_{i+1})\cdot \nu_a>0$. Since 
$\nu_a\in \NS_\mathbb{Q}(X)$, there exists $n$ such that $\ch_1(E_n)\cdot \nu_a=0$ and hence $E_n\notin \T^0_{\nu_a}$. This shows the sequence must terminate.

\end{proof}
Let $\B_{\nu_a, k}$ be the abelian category by tilting $\Coh^{\nu_a}$ at $(\T^1_{\nu_a, k}, \F^1_{\nu_a, k})$, i.e. $\B_{\nu_a, k}=\langle\F^1_{\nu_a, k}[1], \T^1_{\nu_a, k}\rangle$ .  Note that Proposition 10.6 in \cite{CLSY2} implies that $B_{\nu_a, k}$ can be obtained by tilting $\Coh(X)$ at the torsion pair $(\T_{\nu_a, k}, \F_{\nu_a, k})$, where 
	\begin{equation*}
		\begin{split}
			\T_{\nu_a, k}&:=\{E\in \Coh(X)|E\in \T^0_{\nu_a} \text{ and } \Hom_{\Coh(X)}(E, \mC^a_{ k})=0\}\\
			\F_{\nu_a, k}&:=\{E\in \Coh(X)|E\in \langle\mC^a_{k}, \F^0_{\nu_a}\rangle\}.
		\end{split}
	\end{equation*}
\begin{prop}
Let $f: \cO_C(k)[1]\to E$ be a nonzero morphism in $\B_{\nu_a, k}$, then $f$ is injective. 
\end{prop}
\begin{proof}
We argue by contradiction. Assume that $f$ is not injective. Then there is a s.e.s. in $\B_{\nu_a, k}$
\[
0\to \ker(f)\to \cO_C(k)[1]\to \im(f)\to 0.
\]
where $\im(f)[-1]\in \F_{\nu_a, k}$. Hence $\im(f)[-1]$ has no zero dimensional subsheaf.  Taking l.e.s. of cohomology, if $H^{-1}(\ker(f))\neq 0$, we must have $H^{-1}(\ker(f))\simeq \cO_C(k)$. Then $\im(f)[-1]\simeq H^0(\ker(f))\in \T_{\nu_a, k}$, which is a contradiction. So we must have $H^{-1}(\ker(f))=0$, and we have a s.e.s. in $\Coh(X)$
\begin{equation}
\label{eq:5.3ker}
0\to \cO_C(k)\to \im(f)[-1]\to \ker(f)\to 0.
\end{equation}
Since $\im(f)[-1]$ fits into a s.e.s. in $\Coh(X)$
\[
0\to I_1\to \im(f)[-1]\to I_2\to 0,
\]
where $I_1\in \mC^a_{k}$ and $I_2\in \F^0_{\nu_a}$, we must have $\cO_C(k)\hookrightarrow I_1$. Let $Q$ be the cokernel, then $\ker(f)\in\langle Q, I_2\rangle$.  Since $\ker(f)\in \T_{\nu_a, k}$, we have $\Hom(\ker(f), I_2)=0$, hence $I_2=0$ and $\im(f)[-1]\simeq I_1$.
Note that $I_1$ admits a filtration 
\[
0=J_0\subseteq J_1\subseteq J_2\subseteq...\subseteq J_n=I_1
\]
where $J_i/J_{i-1}$'s are subobject of $\cO_C(k)$ in $\Coh^{\nu_a}$. Taking the smallest $i$ such that $\Hom(\cO_C(k), J_i)\neq 0$. Then $\cO_C(k)\simeq J_i/J_{i-1}$, and $\cO_C(k)$ is a direct summand of $J_i$. Then \ref{eq:5.3ker} implies that $\ker(f)\in \mC^a_k$, contradicting $\Hom(\ker(f), \mC^a_k)=0$. Hence $\ker(f)=0$ and $f$ is injective.
\end{proof}
\begin{cor}\label{cor:torsionpairOc(-1)}
The pair $(\T^o_k, \F^o_k):=(\langle\cO_C(k)[1]\rangle, \langle\cO_C(k)[1]\rangle^\perp)$ defines a torsion pair on $\B_{\nu_a, k}$.
\end{cor}

Let $\B^o_{\nu_a, k}:=\langle\F^o_k, \T^o_k[-1]\rangle$ be the heart obtained by tiling $\B_{\nu_a, k}$ at the torsion pair in Corollary \ref{cor:torsionpairOc(-1)}.
\begin{prop}\label{prop:finjective2}
Let $f: \cO_C(k)\to E$ be a nonzero morphism in $\B^o_{\nu_a, k}$, then $f$ is injective.  
\end{prop}
\begin{proof}
Assume $f$ is not injective, we have the s.e.s. in $\B^o_{\nu_a, k}$: 
\[
0\to \ker(f)\to \cO_C(k)\to \im(f)\to 0.
\]
Here $\ker(f)$ 
fits into s.e.s' in $\B^o_{\nu_a, k}$:
\[
0\to K_1\to \ker(f)\to K_2\to 0,
\]
where $K_1$ 
is in $\F^o_k$, $K_2$ 
is in $\T^o_k[-1]$. 
By construction $K_1\in \T_{\nu_a, k}$, hence $\Hom(K_1, \cO_C(k))=0$.
This forces $K_1=0$ and $\ker(f)\in \langle \cO_C(k)\rangle$, which leads to a contradiction.
\end{proof}

\begin{cor}\label{cor:torpairOc2}
The pair $(\T^c_{k}, \F^c_{k}):=(\langle\cO_C(k)\rangle, \langle\cO_C(k)\rangle^\perp)$ defines a torsion pair on $\B^o_{\nu_a, k}$.
\end{cor}

Let $\A_{\nu_a, k}:=\langle\F^c_k, \T^c_k[-1]\rangle$ be the heart obtained by tilting $\B^o_{\nu_a, k}$ at the torsion pair in Corollary \ref{cor:torpairOc2}.


\begin{thm}
\label{thm:corrheartnua}
Using the above notation, we have $\ST(\B^o_{\nu_a, -1})=\A_{\nu_a, -1}$.
\end{thm}
\begin{proof}
The proof is similar to the proof of Theorem \ref{thm:corrheart}, we only point out the difference. 
Note that any $E\in \B^o_{\nu_a, -1}$ fits into a s.e.s in $\B^o_{\nu_a, -1}$:
\[
0\to E'\to E\to \cO_C(-1)^m\to 0 
\]
where $E'\in \F^o_{-1}$. Since $\ST(\cO_C(-1))=\cO_C(-1)[-1]\in \A_{\nu_a, -1}$, we only need to consider $E'$.

Then $E'$ fits into a s.e.s. in $\B_{\nu_a, -1}$
\[
0\to F[1]\to E'\to G\to 0
\]
where $F\in \F_{\nu_a, -1}$ and $G\in \T_{\nu_a, -1}$.
 We first consider $G$. 
 Denote dim $\Ext^i(\cO_C(-1), G)$ by $n_i$. Then by the definition of spherical twist and taking l.e.s. of cohomology, we have 
\[
\begin{split}
0\to H^{-1}(\ST(G))\to \cO_C(-1)^{n_0}&\to G\to H^0(\ST(G))\to \cO_C(-1)^{n_1}\to 0,\\
H^1(\ST(G))&\simeq \cO_C(-1)^{n_2}.
\end{split}
\]
Denote the image of $G\to H^0(\ST(G))$ by $G'$.
Since $G\in \T_{\nu_a, -1}$, any quotient of $G$ in $\Coh(X)$ is also in $\T_{\nu_a, -1}$, hence $G'\in \T_{\nu_a, -1}\subseteq \B_{\nu_a, -1}$. Furthermore $\Hom(\cO_C(-1)[1], G')=0$, which implies that $G'\in \F^o_{-1}\subseteq \B^o_{\nu_a, -1}$ and hence $H^{0}(\ST(G))\in \B^o_{\nu_a, -1}$. Since $H^{-1}(\ST(G))$ is a subsheaf of $\cO_{C}(-1)^{n_0}$, we have $H^{-1}(\ST(G))\in \mC^a_{-1}$. 
By definition of $n_0$, we have 
\[
\Hom(\cO_{C}(-1)[1], H^{-1}(\ST(G))[1])=0,
\]
hence $H^{-1}(\ST(G))[1]\in \F^o_{-1}\subseteq\B^o_{\nu_a, -1}$. From the above discussion, we have $\ST(G)^{\leq 0}\in \B^o_{\nu_a, -1}$.

By the same reasoning as in the Theorem \ref{thm:corrheart}, we have 
 $\Hom(\cO_C(-1), \ST(G)^{\leq 0})=0$, and hence $\ST(G)^{\leq 0}\in \F^c_{-1}$. Also, the l.e.s. of cohomology implies that $H^1(\ST(G))[-1]\in \T^c_{-1}[-1]$, hence we have $\ST(G)\in \A_{\nu_a, -1}$.

Next we consider when $F[1]\in \F_{\nu_a, -1}[1]$. The condition $\Hom(\cO_C(-1)[1], E')=0$ implies that $\Hom(\cO_C(-1)[1], F[1])=0$.
The l.e.s. of cohomology becomes
\[
\begin{split}
0\to F\to H^{-1}(\ST(F[1]))&\to \cO_C(-1)^{n_0}\to 0\\
H^0(\ST(F[1]))&\simeq \cO_C(-1)^{n_1}.
\end{split}
\]
Then $H^0(\ST(F[1]))\in\B^o_{\nu_a, -1}$. 
Consider $H^{-1}(\ST(F[1]))$. 
In the exact triangle 
\[
\cO_{C}(-1)^{n_0}\xrightarrow{\iota} F[1]\to H^{-1}(\ST(F[1]))[1]\xrightarrow{+1}
\]
we have $\cO_{C}(-1)^{n_0}$and $F[1]$ are objects in $\B^o_{\nu_a, -1}$. Then the argument of Proposition \ref{prop:finjective2} implies that $\iota$ is injective in $\B^o_{\nu_a, -1}$. This implies that $H^{-1}(\ST(F[1]))[1]\in \B^o_{\nu_a, -1}$.
Next we compute
\[
\Hom(\cO_C(-1), \ST(F[1]))=\Hom(\ST^{-1}(\cO_C(-1)), F[1])=\Hom(\cO_C(-1)[1], F[1])=0.
\]
Hence $\ST(F[1])\in \F^c_{-1}\subseteq \A_{\nu_a, -1}$.

The above discussion showed that $\ST(E')\in \A_{\nu_a, -1}$, hence $\ST(\B^o_{\nu_a, -1})\subseteq \A_{\nu_a, -1}$.
\end{proof}
Theorem \ref{thm:corrheartnua} implies that we have the following diagram demonstrating the relation of the hearts:
\begin{equation}
\label{eq:tiltdiag}
\begin{tikzcd}
\Coh^{\nu_a}\ar{r}{(\T^1_{\nu_a, -1}, \F^1_{\nu_a, -1})} &[6em]\B_{\nu_a, -1}\ar{r}{(\T^o_{-1}, \F^o_{-1})} &[6em]\B^o_{\nu_a, -1}\ar{d}{(\T^c_{-1}, \F^c_{-1})} \\
\ST_{\cO_C(-1)}(\Coh^{\nu_a})\ar{r}{{\STk(\T^1_{\nu_a, -1}, \F^1_{\nu_a, -1})}}&\ST_{\cO_C(-1)}(\B_{\nu_a, -1})\ar{r}{\STk(\T^o_{-1}, \F^o_{-1})} &\A_{\nu_a, -1}=\ST_{\cO_C(-1)}(\B^o_{\nu_a, -1})
\end{tikzcd}
\end{equation}
The arrows illustrates tilting at the corresponding torsion pairs. 
In particular $\ST_{\cO_C(-1)}(\Coh^{\nu_a})$ can be obtained from $\Coh^{\nu_a}$ via a sequence of tilting. To complete the picture, we give an explicit description of torsion pairs for obtaining $\STk(\Coh^{\nu_a})$ from $\Coh^{\nu_a}$, in particular, we define torsion pairs in the row
\begin{equation}
\label{eq:STtorsionpair}
\begin{tikzcd}
\ST_{\cO_C(-1)}(\Coh^{\nu_a})&[6em]\ar{l}{(\T^s_{\nu_a, -1}, \F^s_{\nu_a, -1})}\ST_{\cO_C(-1)}(\B_{\nu_a, -1}) &[6em]\ar{l}{(\T^s_{-1}, \F^s_{-1})}\A_{\nu_a, -1}
\end{tikzcd}
\end{equation}

Let $Z_{V, \zeta}$ be the central charge solving the equation 
\[
Z_{V, \zeta}(\ST(E))=g_0Z_{V, \nu_a}(E)
\]
where 
 \[
 g_0=\begin{pmatrix}
		1 & 0\\ 
		0 & \frac{1}{D_{\nu_a}+1}
	\end{pmatrix}.
 \]
\begin{prop}
\label{prop:STtorsionpair}
Consider the sequence of tilting in \ref{eq:STtorsionpair}.

(i) The torsion pair $(\T^s_{-1}, \F^s_{-1})$ on $\A_{\nu_a, -1}$ is defined by
\[
\begin{split}
\T^s_{-1}&=\{E\in \A_{\nu_a, -1}\vert \Hom(E, \cO_C(-1))=0\}\\
\F^s_{-1}&=\langle\cO_C(-1)\rangle.
\end{split}
\]
(ii) The torsion pair $(\T^s_{\nu_a, -1}, \F^s_{\nu_a, -1})$ on $\ST(\B_{\nu_a, -1})$ is defined by 
\[
\begin{split}
\T^s_{\nu_a, -1}&=\{E\in\ST(\B_{\nu_a, -1})\vert \forall E\twoheadrightarrow F\in \ST(\B_{\nu_a, -1}), \Im Z_{V, \zeta}(F)< 0\}\\
\F^s_{\nu_a, -1}&=(\T^s_{\nu_a, -1})^\perp.
\end{split}
\]
\end{prop}
\begin{proof}
Consider (i), since 
\[
(\ST(\T^o_{-1}), \ST(\F^o_{-1}))=(\langle\cO_C(-1)\rangle, \langle\cO_C(-1)\rangle^\perp),
\]
and $\A_{{\nu_a}, -1}=\langle\ST(\F^o_{-1}), \ST(\T^o_{-1})[-1]\rangle$, we naturally have 
\[
\begin{split}
\F^s_{-1}&=\ST(\T^o_{-1})[-1]=\langle\cO_C(-1)[-1]\rangle,\\
\T^s_{-1}&=\{T\in \A_{{\nu_a}, -1}\vert \Hom(T, \F^s_{-1})=0\}.
\end{split}
\]
Next we consider (ii).
Let $\T_2=\F^1_{\nu_a, -1}[1]$ and $\F_2=\T^1_{\nu_a, k}$. 
Since $Z_{V, \zeta}(\ST(E))=g_0\cdot Z_{V, \nu_a}(E)$, and $D_{\nu_a}>0$, it is enough to show 
\begin{equation}
\label{eq:T2eq}
\T_2=\{E\in \B_{\nu_a, -1}\vert \forall E\twoheadrightarrow F \in \B_{\nu_a, -1}, \Im Z_{V, \nu_a}(F)< 0.\}
\end{equation}
If $E\in \T_2$, for any $E\twoheadrightarrow F$ in $\B_{\nu_a, -1}$, $F$ is also in $\T_2$, i.e. $F\in \F^1_{\nu_a, -1}[1]$. Then $\Im Z_{V, \nu_a}(F)< 0$.
Conversely, if $E\notin \T_2$, then $E$ fits into a s.e.s. in $\B_{\nu_a, -1}$:
\[
0\to E_1\to E\to E_2\to 0,
\]
where $E_1\in \F^1_{\nu_a, -1}[1]$, $E_2\in \T^1_{\nu_a, -1}$ and $E_2\neq 0$. But $\Im Z_{V, \nu_a}(E_2)\geq 0$, hence $E$ is not an element of the RHS of \ref{eq:T2eq}.
\end{proof}
Define $g\in \tilde{\text{GL}}^+(2, \mathbb{R})=(g_0, f)$, where $f$ is determined by the property that for any $0<\phi\leq 1$, $0<f(\phi)\leq 1$.
\begin{cor}
Let $\A^{\nu_a}$ be the abelian category obtain by tilting $\Coh^{\nu_a}$ at the sequence of torsion pairs in the first row of \ref{eq:tiltdiag}, followed by \ref{eq:STtorsionpair}. Then $(Z_{V, \zeta}, \A^{\nu_a})$ defines a Bridgeland stability conditions and 
\[
(Z_{V, \zeta}, \A^{\nu_a})\cdot g=\STk\cdot(Z_{V, \nu_a}, \Coh^{\nu_a}).
\]
\end{cor}
\begin{proof}
This is a combination of Theorem \ref{thm:corrheartnua} and Proposition \ref{prop:STtorsionpair}.
\end{proof}
\begin{rem}
\label{rk:5.9nonnef}
Let $C$, $\nu$ and $D$ be as in the first paragraph of this Section. 
Take $\nu_a=\nu+aD$ for some $a\in \mathbb{Q}_{>0}$. Assuming $D$ is nef, then $\nu_a$ is ample. 
Then the Bridgeland stability condition $\ST_{\cO_C(-1)}(\sigma_{V, \nu_a})$ is of the form $\sigma_{V, \nu_b}=(Z_{V, \nu_b}, \ST_{\cO_C(-1)}(\Coh^{\nu_a}))$. 
From the computation in Section 4.1, we have $\nu_b=\nu+bD$ where $b=-\frac{a}{a+1}$. Then 
Proposition \ref{prop:STtorsionpair} gives an explicit construction of a Bridgeland stability condition associated to a non-nef divisor $\nu_b$. 
\end{rem}
This suggests that the Bridgeland stability condition $\STk(\sigma_{V, \nu_a})$ may not be in the closure of the geometric chamber of the stability manifold. We verify this by looking at the stability of skyscraper sheaves. The following computation is well known and follows directly from GRR computations. 
\begin{prop}
\label{Lem:imSTOp}
Let $p$ be a point in $X$. 

(i) If $p\notin C$, then $\STt^{-1}(\cO_p)=\cO_p$ and $\STt(\cO_p)=\cO_p$.

(ii) If $p\in C$, then $\STt^{-1}(\cO_p)\simeq E^\bullet$, such that $H^{-1}(E)\simeq \cO_C(t)$, $H^0(E)\simeq \cO_C(t+1)$, and $H^j(E)=0$ for $j\neq 0, -1$.
Also, $\STt(\cO_p)\simeq F^\bullet$, such that $H^{-1}(F)\simeq \cO_C(t-1)$, $H^0(F)\simeq \cO_C(t)$, and $H^j(F)=0$ for $j\neq 0, -1$.

 
\end{prop}
Let $p$ be a point in $X$ such that $p\in C$. Then Proposition \ref{Lem:imSTOp} implies that $\ST^{-1}_{\cO_C(-1)}(\cO_p)$ is not in $\Coh^{\nu_a}[n]$ for any $n$. Hence $\ST^{-1}_{\cO_C(-1)}(\cO_p)$ is not $\sigma_{V, \nu_a}$-semistable, which implies that $\cO_p$ is not  $\ST_{\cO_C(-1)}\cdot\sigma_{V, \nu_a}$-semistable. 

\bigskip
We finish this Section by an observation that the functor $\ST_{\cO_C(-1)}$ is equipped with a moduli interpretation. 
Recall in \cite{tramel2017bridgeland}, the author constructed a moduli space $\mathcal{M}_\tau([\cO_x])$
of $\tau$-stable objects of class $[\cO_x]$, and shows that it is isomorphic to $Y:=X\sqcup_C\PP^1$. Since $Y\setminus \PP^1\simeq X\setminus C$, hence $Y$ is smooth at any point q such that $q\notin \PP^1$. For a point $p\in \PP^1$, Lemma 7.10 in \cite{tramel2017bridgeland} implies that
$T_pY\simeq \Ext^1(E_p, E_p)\simeq \mathbb{C}^2$. Here $E_p$ is the object parametrized by the point $p$. This implies that $Y$ is smooth and hence $X\simeq Y$. Let $\Phi: D^b(Y)\to D^b(X)$ be the Fourier-Mukai transform with kernel given by the universal object $\mathcal{U}$ on $Y\times X$. 
\begin{prop}
We have an equivalence of functors $\Phi=\STk^{-1}$.
    \end{prop}
\begin{proof}
By Proposition 7.9 in \cite{tramel2017bridgeland}, we have
\[
\mathcal{U}[1]\simeq\Cone(\cO_{\Delta_X}\to \cO_{C\times C}(-1, -1)[1]).
\]
Proposition 1.1.6 in \cite{HL} implies that $\mathcal{E}xt^i(\cO_{\Delta_X}, \cO_{X\times X})=0$ for $i=0, 1$, and $\mathcal{E}xt^3(\cO_{\Delta_X}, \cO_{X\times X})$ is supported in dimension $0$. Since $\Ext^3(\cO_{\Delta_X}, \cO_{X\times X})=0$, local to global spectral sequence implies that $\mathcal{E}xt^3(\cO_{\Delta_X}, \cO_{X\times X})=0$. The proof of 1.1.8 in \cite{HL} also implies that $\cO_{\Delta X}^\vee\simeq \mathcal{E}xt^2(\cO_{\Delta_X}, \cO_{X\times X})$ is pure. 
Then the Chern character forces $\cO_{\Delta_X}^\vee\simeq \cO_{\Delta_X}[-2]$.
Similar argument shows that $\cO_C(-1)^\vee\simeq \cO_C(-1)[-1]$, hence we have $\cO_{C\times C}(-1, -1)^\vee\simeq \cO_{C\times C}(-1, -1)[-2]$. 
Hence, we obtain
\[
\begin{split}
\mathcal{U}[-2]&\simeq \Cone(\cO_{\Delta X}^\vee\to \cO_{C\times C}(-1, -1)^\vee[1])[-1]\\
&\simeq \Cone(\cO_{C\times C}(-1, -1)[-1]\to \cO_{\Delta_X})^\vee.
\end{split}
\]
Recall in Definition \ref{def:ST}, $\STk$ is a Fourier-Mukai transform with kernel equals to 
\[
\Cone(\cO_{C\times C}(-1, -1)[-1]\to \cO_{\Delta X}),
\]
hence $\mathcal{U}$ is the kernel of $\STk^{-1}$. Since both functors are Fourier-Mukai transforms, there is an equivalence of functors $\Phi=\STk^{-1}$. In particular this also implies that $\Phi$ is an equivalence. 
\end{proof}

\section{Stability of some line bundles at $\sigma^b_{V, \nu}$}

In this Section,
we study the stability of certain line bundles at $\sigma^b_{V, \nu}$. Let $X$, $\nu$, $C$ be as in Section 4. 

\begin{lem}\label{lem:AG52-122-1}
\label{lem:wallestimate}
Let $X$ be a smooth projective surface,  $\nu$ a nef class on $X$, and $L$ a line bundle on $X$ such that  
$L\in \B_{\nu, -2}$. Then there exists a $V_o$ such that for any $\sigma^b_{V, \nu}$-destabilizing sequence
\begin{equation}\label{eq:destabALQ}
0 \to A \to L \to Q \to 0
\end{equation}
in $\Bc_{\nu, -2}$ where 

(i) $\phi^b_{V, \nu} (A)=\phi^b_{V, \nu}(L)$,

(ii) $A$ is $\sigma^b_{V, \nu}$-semistable,

(iii) $\ch_0(A)\geq 2$,

then $V<V_o$.
\end{lem}
\begin{proof}
We abbreviate the notation and denote $\phi^b_{V, \nu}$ by $\phi$. Let $F$ be the last HN-factor of $A$, and we have 
\[
0\to K\to A\to F\to 0
\]
in $\Coh(X)$. If $\Hom(K, \cO_C(-2))=0$, 
then the surjection $A\to F$ in $\Coh(X)$ is also a surjection in $\B_{\nu, -2}$. 
If $\Hom(K, \cO_C(-2))\neq 0$, then $K$ fits into a s.e.s. in $\Coh^\nu$:
\[
0\to K'\to K\to R\to 0,
\]
where $\Hom(K', \cO_C(-2))=0$ and $R\in \mC_{-2}$.
Composing with the inclusion $K\rightarrow A$, we obtain a s.e.s. in $\Coh(X)$
\begin{equation}
\label{eq:HNreplace}
0\to K'\to A\to F'\to 0.
\end{equation}
Then $F'\in \T^0_\nu$, and $\Hom(A, \mC_{-2})=0$ implies that $\Hom(F', \mC_{-2})=0$. Hence $F'\in \T_{\nu, -2}$. Since $K'\in \T_{\nu, -2}$, we have
\ref{eq:HNreplace} is a s.e.s in $\B_{\nu, -2}$.

Then by assumption, we have $\phi(L)=\phi(A)\leq \phi(F')$. Also we have the s.e.s in $\Coh(X)$:
\[
0\to R\to F'\to F\to 0.
\]
Since $R\in \mC_{-2}$, we have $\ch_2(R)\leq 0$, hence $\phi(F')\leq \phi(F)$.
To summarize, we have
\begin{equation}
\label{eq:compphase}
\begin{split}
\frac{\ch_2(L)-V}{\ch_1(L)\cdot\nu}&\leq \frac{\ch_2(F')-V\ch_0(F')}{\ch_1(F')\cdot\nu}\\
&\leq\frac{\ch_2(F)-V\ch_0(F)}{\ch_1(F)\cdot\nu}.
\end{split}
\end{equation}
By Corollary 4.3 in \cite{CLSY1}, we have $F$ satisfies Bogomolov-Gieseker inequality. Since $\nu$ is nef and big, HIT also holds, so we have 
\begin{equation}
\label{eq:BGnef}
  \ch_2(F) \leq \frac{\ch_1(F)^2}{2\ch_0(F)} \leq \frac{1}{2\ch_0(F)} \frac{ (\nu \ch_1(F))^2}{\nu^2}.
\end{equation}
Combining Equation \ref{eq:compphase} and \ref{eq:BGnef}, we have 
\[
\frac{\ch_2(L)-V}{\ch_1(L)\cdot\nu}\mu_\nu(F)\leq \frac{1}{2\nu^2}\mu_\nu(F)^2-V.
\]
Since $\mu_\nu(F)>0$, this quadratic inequality in $\mu_\nu(F)$ implies that 
\[
\mu_\nu(F)\geq\nu^2(\sqrt{(\frac{\ch_2(L)-V}{\ch_1(L)\cdot\nu})^2+\frac{2V}{\nu^2}}+\frac{\ch_2(L)-V}{\ch_1(L)\cdot\nu}).
\]
Let $\alpha:=\ch_1(L)$. Since  $\mu_\nu(A)\geq\mu_\nu(F)$, we have 
\begin{equation}
\label{eq:lem6.6boundrk}
\ch_0(A)\leq\frac{\alpha\cdot\nu}{\nu^2(\sqrt{(\frac{\frac{1}{2}\alpha^2-V}{\alpha\cdot\nu})^2+\frac{2}{\nu^2}V}+\frac{\frac{1}{2}\alpha^2-V}{\alpha\cdot\nu})}
\end{equation}
The limit of the RHS approaches $1$ as $V\to \infty$. 
Then there exists $V_o>0$, such that for $V>V_o$ we have $\ch_0(A)\leq1$. 
\end{proof}

\begin{defn}
\label{def:destwall}
We call a value $V_0$ in the ray $V\cdot\nu$ a destabilizing wall if there exists a destabilizing sequence
\begin{equation*}
0 \to A \to L \to Q \to 0
\end{equation*}
in $\Bc_{\nu, -2}$ such that $A$ satisfies 
(i') $\phi_{V_0, \nu} (A)=\phi_{V_0, \nu}(L)=\phi_{V_0, \nu}(Q)$ and
(ii) in Lemma \ref{lem:AG52-122-1}. 
\end{defn}
Using the same argument of Proposition 3.10 in \cite{LQ}, we have the set of destabilizing walls $W(A, L)$ is locally finite for $\nu\in \NS_{\mathbb{Q}}(X)$.
By the previous lemma, the highest destabilizing wall for higher rank object exists, and we denote it by $V_h$.  
\begin{prop}
\label{prop:Hom(,OC(-1))=0}
Let $L$ be a line bundle on $X$, then for each $E\hookrightarrow L$ in $\B_{\nu, -2}$, there exists $E'\hookrightarrow L$ in $\B_{\nu, -2}$ such that $\Hom(E', \cO_C(-1))=0$ and $\sigma^b_{V, \nu}(E)=\sigma^b_{V, \nu}(E')$. 
\end{prop}
\begin{proof}
Since $E\in \B_{\nu, -2}$, $\Hom(E, \cO_C(-2))=0$. If $\Hom(E, \cO_C(-1))\neq 0$, any map from $E$ to $\cO_C(-1)$ must be surjective in $\Coh(X)$. We complete to a s.e.s. in $\Coh(X)$:
\begin{equation}
\label{eq:seskerOc(-1)}
0\to E_1\to E\to \cO_C(-1)\to 0.
\end{equation}
Consider the following commutative diagram formed by pullout:
\[
\begin{tikzcd}
    0 \ar{r}&E_1\ar{r}\ar{d} & E\ar{r}\ar{d}& \cO_C(-1)\ar{r}\ar{d} & 0\\
    0 \ar{r}&\HN_n(E_1)\ar{r}& P\ar{r}& R\ar{r} & 0
\end{tikzcd}
\]
 where $\HN_n(E_1)$ is the last HN factor of $E_1$, $P$ is formed by pushout diagram, then $R$ is a quotient sheaf of $\cO_C(-1)$. Since $\T_{\nu, -2}$ is a torsion class in $\Coh(X)$, we have $P\in \T_{\nu, -2}$. Then $\mu_\nu(R)=0$ implies that $\mu_\nu(\HN_n(E_1))>0$.
Since $\Hom(E, \cO_C(-2))=0$, we also have $\Hom(E_1, \cO_C(-2))=0$, hence $E_1\in \T_{\nu, -2}$ and equation \ref{eq:seskerOc(-1)} is also a s.e.s. in $\B_{\nu, -2}$. Repeating this process, we obtain $E'\hookrightarrow L$ in $\B_{\nu, -2}$ and $\Hom(E', \cO_C(-1))=0$. Since $\cO_C(-1)\in \ker Z_{V, \nu}$, we have $\sigma^b_{V, \nu}(E)=\sigma^b_{V, \nu}(E')$.
\end{proof}
\begin{prop}
\label{prop:lowbdC}
Let $L$ be a line bundle on $X$ with $c=0$.
Let $F\hookrightarrow L$ in $\B_{\nu, -2}$, then there exists $0\to E\to L\to Q\to 0$ in $\B_{\nu, -2}$ such that $C\cdot \ch_1(H^0(Q))\geq -2$ and $\sigma^b_{V, \nu}(E)=\sigma^b_{V, \nu}(F)$.
\end{prop}
\begin{proof}
By the Proposition \ref{prop:Hom(,OC(-1))=0}, there exists $0\to E\to L\to Q\to 0$ in $\B_{\nu, -2}$ such that $\Hom(E, \cO_C(-1))=0$ and $\sigma^b_{V, \nu}(E)=\sigma^b_{V, \nu}(F)$. Since $H^0(Q)$ is a torsion sheaf, we can write $\ch_1(H^0(Q))=a_0C+\sum a_iC_i$ where $a_i\in \mathbb{Z}_{\geq 0}$ and $C_i$'s are irreducible curves in $X$ which does not equal to $C$. 
Then $L|_C\twoheadrightarrow H^0(Q)|_C$ implies that $H^0(Q)|_C\simeq \cO_C$ or supported in dimension $0$. If $a_0>0$, we have 
\[
0\to K_1\to H^0(Q)\to \cO_C\to 0,
\]
where $L(-C)\twoheadrightarrow K_1$. If nonzero, $K_1|_C\simeq \cO_C(2)$ or is supported on dimension $0$. If $K_1|_C\simeq \cO_C(2)$, denote $\ker(K_1\to \cO_C(2))$ by $K_2$ and repeat this process, we get $K_{a_0-1}|_C\simeq \cO_C(2(a_0-1))$, and $K_{a_0}=0$ or supported in dimension $0$. Then we have $H^0(Q)$ fits into a s.e.s. 
\[
0\to M\to H^0(Q)\to N\to 0,
\]
where $M|_C$ is zero or supported in dimension zero, $N\in \langle\cO_C(i)|i\geq 0\rangle$. Riemann-Roch computation implies that $\Ext^1(\cO_C(i), \cO_C(-1))\simeq \mathbb{C}^{i}$. 

Claim. If $a_0>1$, then $\dim\Ext^1(H^0(Q), \cO_C(-1))\geq 1$. 

Indeed, it is enough to show that $\dim\Ext^1(N, \cO_C(-1))\geq 1$. Hence it is enough if we have $\dim\Ext^1(N_1, \cO_C(-1))\geq 1$, where $N_1$ fits into the s.e.s. in $\Coh(X)$
\[
0\to \cO_C(2)\to N_1\to \cO_C\to 0.
\]
Let $p\in X$ such that $\Hom(N_1, \cO_p)\simeq \mathbb{C}^2$. Apply $\Hom(\_, N_1)$ to the s.e.s. 
\[
0\to \cO_C(-1)\to \cO_C\to \cO_p\to 0. 
\]
Since $\Ext^1(\cO_C, N_1)=0$, we have 
\[
0\to \Ext^1(\cO_C(-1), N_1)\to \Ext^2(\cO_p, N_1)\to \Ext^2(\cO_C, N_1)\to 0.
\]
Hence $\dim\Ext^1(\cO_C(-1), N_1)=1$.

Denoting the image of $E\to L$ by $L_n(-C')$. Then $\Hom(E, \cO_C(-1))=0$ implies the vanishing $\Hom(L_n(-C'), \cO_C(-1))=0$ and hence $\Ext^1(H^0(Q), \cO_C(-1))=0$. This leads to a contradiction, hence $a_0\leq 1$. 
\end{proof}

\begin{cor} 
\label{prop:C.C'>1}
Let $X$ be a smooth K3 surfaces, $C$ a $(-2)$ curve on $X$, and $L$ be a line bundle on $X$ such that $c=0$. 
Assume that for any irreducible curve $C''\subseteq X$, $C''\neq C$, we have $C''\cdot C> 2$. 
Then $L$ is $\sigma^b_{V, \nu}$-semistable for all $V>0$.
\end{cor}

\begin{proof}
Assume that $E$ destabilizes $L$ at $\sigma^b_{V, \nu}$ for some $V$. 
Let $L_n(-C')$ be the image of $E\to L$. By Proposition \ref{prop:Hom(,OC(-1))=0}, we can assume that $\Hom(E, \cO_C(-1))=0$, hence $\Hom(L(-C'), \cO_C(-1))=0$. This implies that $C'\cdot C\leq 0$. On the other hand, Proposition \ref{prop:lowbdC} implies that $a_0\leq 1$. Combine these two results, for any irreducible component $C''$ of $C'$, $C''\neq C$, we have $0\leq C''\cdot C\leq 2$. By assumption of the Proposition, such $C''$ does not exist, hence $C'=C$. 

Since $\phi^b_{V, \nu}(L_n(-C))\leq \phi^b_{V, \nu}(L)$ for any $V$, we have rank of $E$ is greater than $1$. By the discussion after Definition \ref{def:destwall}, the highest destabilizing wall for higher rank object exists. Let $F\rightarrow L$ in $\B_{\nu, -2}$ such that $W(F, L)=V_h$. Write $\ch_1(L)=m\nu+n H$, where $H\cdot \nu=0$. Consider the space of Bridgeland/weak stability conditions $\sigma^b_{\nu, B}$, where $B=y\nu+zH$. Note that  $c=0$ implies that $H\cdot C=0$ and $B\cdot C=0$ for any $y, z$. Hence $\sigma^b_{\nu, B}$ has the heart $\B^y_{\nu, -2}$. 
Furthermore, since $W(F, L)=V_h$, 
$F$ must be torsion free, and hence a suboject of $L$ in $\Coh^y$. Indeed, if $F$ is not torsion free, denoting the torsion subsheaf of $F$ by $F_0$, then $F/F_0$ is a subobject of $L$ in $\B^y_{\nu, -2}$, and $W(F/F_0, L)_z$ is higher than $W(F, L)_z$, contradicting $W(F, L)$ is the highest wall at $(y, z)$.
The same argument in \cite{ARCARA20161655} implies that $W(F, L)$ is lower than $W(L_n(-C), L)$ for $|z|\gg 0$, which contradicts $\phi^b_{V, \nu, B}(L_n(-C))<\phi^b_{V, \nu, B}(L)$ for any $y, z$. Hence such $F$ does not exists. 
\end{proof}

The K3 surfaces satisfies the assumption in the Corollary \ref{prop:C.C'>1} exists.
\begin{eg}
Consider the case Theorem 2(v) in \cite{Kov94}. 
Let $\rho(X)=2$. 
Then the effective cone of $X$ is generated by two smooth rational curves of self intersection $-2$. Denoting the two curves by $C_1$ and $C_2$ and $C_1\cdot C_2=q$, then $q>2$. Since otherwise, the first Chern character of any line bundle has non-positive self-intersection, which implies that there is no ample line bundle on $X$, a contradiction. In this case, take $y$ to be an even positive integer, and $\nu=\frac{q}{2}yC_1+yC_2$, then $\nu$ satisfies $\nu\cdot C_1=0$, $\nu\cdot C_2>0$ and $\nu^2>0$. 
\end{eg}
\begin{eg}
Again consider the case Theorem 2(v) in \cite{Kov94}. 
Let $\rho(X)=3$. Assume that the effective cone of $X$ is generated by three curves $C_1$, $C_2$ and $C_3$. Then Corollary 2.9 (i) in \cite{Mor84} implies that there exists K3 surface such that the pairwise intersection between $C_i$'s are greater than $2$. 
\end{eg}

\bibliography{refs}{}
\bibliographystyle{plain}

\end{document}